\documentclass[a4paper,12pt]							{amsart}
\usepackage									{ellipsis} 							
\usepackage									{lmodern} 							
\usepackage									{amssymb, amsmath, amsthm}			
\usepackage									{bm, bbm, dsfont}
\usepackage[fixamsmath]						{mathtools} 							
\usepackage[shortlabels]					{enumitem} 							
\usepackage[spacing,kerning]				{microtype}
\usepackage									{geometry}
\usepackage 								{hyperref}
\usepackage									{csquotes}							
\usepackage									{todonotes,verbatim}
\usepackage{subcaption}
\usepackage{pdfsync}

\allowdisplaybreaks
\microtypecontext{spacing=nonfrench}

\theoremstyle{plain}
\newtheorem{thm}					{Theorem}[section]
\newtheorem{lemma}		[thm]		{Lemma}
\newtheorem{cor}			[thm]	{Corollary}
\newtheorem{proposition}	[thm]	{Proposition}

\theoremstyle{definition}

\theoremstyle{remark}

\newtheorem{rem}			[thm]	{Remark}
\newtheorem*{exa*}					{Example}
\newtheorem*{rem*}					{Remark}

\DeclareMathOperator	{\IN}		{\mathbb{N}} 	
	\DeclareMathOperator	{\IR}		{\mathbb{R}}

\newcommand{\R}{\IR}

\DeclareMathOperator	{\tr}		{tr}

\DeclareMathOperator	{\esssup}	{\esssup}

\DeclarePairedDelimiter	\abs		{\lvert}	{\rvert}

\DeclarePairedDelimiter	\skal		{\langle}	{\rangle}

\newcommand				{\eins}		{\text{$\mathds{1}$}}
\renewcommand 		{\epsilon}	{\varepsilon}
\renewcommand			{\phi}		{\varphi}

\newcommand 		{\vep}	{\varepsilon}
\numberwithin			{equation}{section}

\title{Large deviations and a phase transition in the Block Spin Potts models}

\author[Holger Kn\"opfel]{Holger Kn\"opfel}
\address[Holger Kn\"opfel]{Fachbereich Mathematik und Informatik,
Universit\"at M\"unster,
Einsteinstra\ss e 62,
48149 M\"unster,
Germany}

\email[Holger Kn\"opfel]{Holger.Knoepfel@ruhr-uni-bochum.de}

\author[Matthias L\"owe]{Matthias L\"owe}
\address[Matthias L\"owe]{Fachbereich Mathematik und Informatik,
Universit\"at M\"unster,
Einsteinstra\ss e 62,
48149 M\"unster,
Germany}

\email[Matthias L\"owe]{maloewe@math.uni-muenster.de}

\author[Holger Sambale]{Holger Sambale}
\address[Holger Sambale]{Fakult\"at f\"ur Mathematik,
Universit\"at Bielefeld,
Postfach 100131,
33501 Bielefeld,
Germany}

\email[Holger Sambale]{hsambale@math.uni-bielefeld.de}

\begin{document}
\subjclass{Primary 60F10, Secondary 82B20}
\keywords{block spin Potts model, large deviation principle, logarithmic Sobolev inequality, phase transition}
\thanks{Research of the second author was
funded by the Deutsche Forschungsgemeinschaft (DFG, German Research Foundation) under Germany 's
Excellence Strategy
EXC 2044-390685587, Mathematics M\"unster: Dynamics-Geometry-Structure. Research of the third author was funded by the Deutsche Forschungsgemeinschaft via the CRC 1283 \emph{Taming uncertainty and profiting from randomness and low regularity in analysis, stochastics and their applications}.}

\date{\today}
\begin{abstract}
We introduce and analyze a generalization of the blocks spin Ising (Curie-Weiss) models that were
discussed in a number of recent articles. In these block spin models each spin in one of $s$ blocks
can take one of a finite number of $q \ge 3$ values, hence the name block spin Potts model.
The values a spin can take are called colors. We prove a large deviation principle for the
percentage of spins of a certain color in a certain block. These values are represented in an $s
\times q$ matrix. We show that for uniform block sizes and appropriately chosen interaction
strength
there is a phase transition. In some regime the only equilibrium is the uniform distribution of all
colors in all blocks, while in other parameter regimes there is one predominant color, and this is
the same color with the same frequency for all blocks. Finally, we establish log-Sobolev-type inequalities for the block spin Potts model.
\end{abstract}

\maketitle

\section{Introduction}
Mean-field models as the Curie--Weiss model are a first order approximation of lattice models, yet
they often show qualitatively interesting results (see \cite{Ell06} for a classic survey). In
particular,
mean-field block models have been proposed as an approximation of lattice models for meta-magnets,
see e.\,g.\ \cite{KC75}.
To describe them, assume that we have $N$ interacting particles that carry a spin. Also assume that
we can group these particles into several groups. The interaction is such that particles within the
same group interact with one interaction strength, while particles in different groups have
another,
usually smaller, interaction strength.

In a sequence of papers the statistical mechanics of such models was studied from various points of
view, see \cite{GC08}, \cite{FC11},
\cite{Col14}, \cite{LS18}, \cite{KT18}, \cite{KLSS19}, \cite{KT20}. In particular they were also
discussed as models for social interactions between several groups, e.\,g.\ in \cite{GBC09},
\cite{ABC10}, \cite{OOA18}, \cite{LSV20} (the latter paper studies a combination of sparse Ising
models on Erd\"os--R\'enyi graphs as in \cite{BG93b}, \cite{KLS19}, or \cite{KLS20} and block
models).
Recently, block models have also been studied in a statistical context (see \cite{BRS19},
\cite{LS20}). Here the task is to exactly recover the block structure from a given number of
realizations of the model and it turns out that this can be done surprisingly effectively.

However, all the literature cited above deals with Ising spins, i.\,e.\ the spins take two values
(usually $\pm 1$). Of course, the physics literature knows many more spin models than just the
Ising
model, in particular models with a continuous spin as Heisenberg models and XY-models.

On the discrete side Potts models (cf. e.\,g.\ \cite{WuPotts,
Kesten_Potts,EllisWang-LimitTheoremsForTheEmpiricalVectorOfTheCWPottsModel,EllisTouchette_Potts})
are the most natural generalization of Ising models. For them each particle carries a spin from a
finite set (of cardinality $3$ or larger).

The aim of the present note is to investigate block spin Potts models as a natural generalization
of
block spin Ising models. We will basically concentrate on models where the blocks have
approximately identical size and where the interaction is purely ferromagnetic, i.\,e.\ particles tend
to have the same spin, no matter, whether they are in the same block or in different ones.

Similar to \cite{LS18} and parts of \cite{KLSS19} our main tool are large deviation techniques.
Indeed, as we will see in Section 3, it is not too difficult to establish a large deviation
principle for the ``block magnetizations''. However, to derive a limit theorem with
an explicit limit law from there turns out to be more complicated than in the case of Ising spins
(which is quite a common feature in Potts models).

The rest of this note is organized in the following way. In the next section we will describe the
block spin Potts model. Section 3 contains a large deviation analysis of this model. In
Section 4, we will concentrate on a version with blocks of asymptotically equal size and compute
the
possible limit laws for such models. Finally, in Section 5 we prove (modified) logarithmic Sobolev inequalities for the block spin Potts model and discuss some background and possible applications.

Let us mention at this point that, while we were finishing the current manuscript we learned that in \cite{Liu2020}
the author studies a very similar model: Here the number of blocks is restricted to two, but they may be of different size. His techniques, however are different from ours. Moreover, extending his results,  we prove a large deviation principle, are able to locate the minima of the rate functions and show logarithmic Sobolev inequalities.

\emph{Acknowledgment.} The authors would like to thank Arthur Sinulis for fruitful discussions.

\section{The model}

In the sequel we will consider the following model. Take the set $S=\{1, \ldots, N\}$ and partition
$S$ into $s$ sets $S_1, \ldots, S_s$. These sets will, of course, depend on $N$ and we assume that
the limits $\gamma_k:= \lim_{N \to \infty}\frac{|S_k|}{N} \in (0,1)$ exist (and, of course,
$\sum_{k=1}^s \gamma_k=1$).

Moreover, take an integer $ q \ge 3$ and for $\omega \in \{1,\ldots,q \}^S$ and
$0<\alpha < \beta$ introduce the Hamiltonian
$$
H_{N,\alpha,\beta}(\omega):=H_N(\omega):= -\frac{\beta}{2N} \sum_{i \sim j}
\eins_{\omega_i=\omega_j}
-\frac{\alpha}{2N} \sum_{i \not\sim j}  \eins_{\omega_i=\omega_j}.
$$
Here $i \sim j$ means that the indices $i$ and $j$ belong to the same block $S_k$ (where the case
$i=j$ is included) for some $k\in \{1, \ldots, s\}$, while we write $i \not\sim j$, if this is not
the case. With $H_{N,\alpha,\beta}$ we will associate the Gibbs measure
$$
\mu_{N,\alpha,\beta}(\omega):=\mu_N(\omega):= \frac{\exp(-H_N(\omega))}{Z_{N,\alpha,\beta}}
$$
where, of course,
$$
Z_{N,\alpha,\beta}:= Z_N:= \sum_{\omega'} \exp(-H_N(\omega')).
$$
For $k \in \{1, \ldots, s\}$ and $c \in \{1, \ldots, q\}$ denote by $m_{k,c}$ the relative number
of spins of ``color'' $c$ in the block $S_k$, i.\,e.\
$$
m_{k,c}:=m_{k,c}(\omega):=\frac 1 {|S_k|} \sum_{i \in S_k} \eins_{\omega_i=c},
$$
and set $M_N:=(m_{k,c})\in M(s\times q)$. Note that $M_N$ is an order parameter of
the model in the sense that the Hamiltonian is a function of $M_N$ rather than $\omega$.

Indeed, since
$\eins_{\omega_i = \omega_j}=\sum_{c=1}^q \eins_{\omega_i = c} \eins_{\omega_j = c}$ we have
\begin{align}\label{representation H}
\begin{split}
-2N H_N(\omega) &=
\beta \sum_{c = 1}^q \sum_{k = 1}^s \sum_{i \in S_k} \sum_{j \in S_k} \eins_{\omega_i = c}
\eins_{\omega_j = c}
+ \alpha \sum_{c = 1}^q \sum_{k \neq k'}
\sum_{i \in S_k} \sum_{j \in S_{k'}} \eins_{\omega_i = c}
\eins_{\omega_j = c}
\\
&=\tr (B^tAB)
\end{split}
\end{align}
where $A_{\alpha,\beta}:=A \in M(s \times s)$ is the symmetric matrix with entries $\beta$ on and
$\alpha$ off the diagonal (the \emph{block interaction matrix}) and $B\in M(s\times q)$ has entries
\begin{equation}\label{matrixB}
 b_{k,c}=\sum_{i \in S_k} \eins_{\omega_i = c}=|S_k|m_{k,c}.
\end{equation}
Now $A$ is positive definite for $0<\alpha<\beta$ due to
$x^tAx=(\beta-\alpha)x^2+\alpha(\sum_kx_k)^2$ for $x\in\R^s$. So using the (unique) positive
definite symmetric matrix $\sqrt A$ we see that
\[
 \tr(B^tAB)=\tr\big((\sqrt{A} B)^t(\sqrt{A}B)\big)=[ \sqrt{A}B, \sqrt{A}B]
\]
denoting by $[\,,\,]$ the Frobenius scalar product. Hence
the Hamiltonian is a positive definite quadratic form of the matrix $B$ and we will write
$\tr(B^tAB)=\skal{B,B}_A$.
Now introducing the diagonal matrix
$\Gamma_N \in M(s \times s)$ given by $(\Gamma_N)_{k,k} = \abs{S_k}$
 we finally rewrite \eqref{representation H} as
\begin{equation}\label{eqn:HamiltonianAsScalarProduct}
H_N(\omega) = -\frac{1}{2N} \skal{\Gamma_N  M_N, \Gamma_N  M_N}_{A}.
\end{equation}
It is therefore natural to study the distribution of $M_N$ under the Gibbs measure $\mu_N$.

\section{\texorpdfstring{A Large Deviation Principle for $M_N$}{A Large Deviation Principle for
M\_N}}
In this section we prove a Large Deviation Principle (LDP) for the matrix $M_N$. The analysis of the
corresponding rate function will help us to determine the limiting behavior of $M_N$ and to prove
the existence of a phase transition.

Let us briefly recall the definition of a large deviation principle (cf. \cite{dH00} and
\cite{DZ10} for a rich survey of many large deviation results): For a Polish space $\mathcal{X}$
and an increasing sequence of non-negative real numbers $(a_n)_{n \in \IN}$ a sequence of
probability measures $(\nu_n)_n$ on $\mathcal{X}$ is said to satisfy a \emph{large deviation
principle} with speed $a_n$ and rate function $I: \mathcal{X} \to \IR$ (by which we mean a lower
semi-continuous function with compact level sets $\{x: I(x) \le L\}$ for all $L>0$), if for all
Borel sets $B \in \mathcal{B}(\mathcal{X})$ we have
\begin{equation*}
	-\inf_{x \in \mathrm{int}(B)} I(x) \le \liminf_{n \to \infty} \frac{\log \nu_n(B)}{a_n} \le
\limsup_{n \to \infty} \frac{\log \nu_n(B)}{a_n} \le -\inf_{x \in \mathrm{cl}(B)} I(x).
\end{equation*}
Here $\mathrm{int}(B)$ and $\mathrm{cl}(B)$ denote the topological interior and closure of a set
$B$, respectively.

We say that a sequence of random variables $X_n: \Omega \to \mathcal{X}$ satisfies an LDP with speed
$a_n$ and rate function $I: \mathcal{X} \to \IR$ under a sequence of measures $\mu_n$ if the
push-forward sequence $\nu_n \coloneqq \mu_n \circ X_n$ satisfies an LDP with speed $a_n$ and rate
function $I$.

Now if $M_N(k)$ denotes the row $k$ of $M_N$ for a fixed $k$, then under the uniform
measure $\rho$ on $\{1, \ldots, q\}$ the vector
$M_N(k)$ is the empirical vector of a $|S_k|$-fold
drawing from the alphabet $\{1, \ldots, q\}$. Thus, under the uniform measure
$\rho^{|S_k|}$ the vector $M_N(k)$ obeys an LDP with speed $|S_k|$ and a rate function that is
given by the relative entropy of a probability measure $\nu \in \mathcal{M}^1(\{1, \ldots, q\})$
with respect to $\rho$
$$
H(\nu|\rho):= \sum_{c=1}^q \nu(c) \log\frac{\nu(c)}{\rho(c)}
$$
(see e.\,g.\ \cite[Theorem 2.1.10]{DZ10} for a reference). Note that
$$
H(\nu|\rho)= \sum_{c=1}^q \nu(c) \log \nu(c)+\sum_{c=1}^q \nu(c) \log q =: H(\nu) + \log q,
$$
where $H(\nu)$ is the entropy of $\nu$ and we adopt the convention that $0 \log 0= 0$.

Now $M_N(k)$ are independent random vectors for $k=1, \ldots, s$. Consequently, for
$\overline \rho_N:= \bigotimes_{k=1}^s
\rho^{|S_k|}$ we have
$$
\frac 1 N \log \overline \rho_N (M_N\in B) = \frac 1 N \sum_{k=1}^s \log  \rho^{|S_k|} (M_N(k) \in
B_k).
$$
for any set $B= \prod_{k=1}^s B_k$ with Borel sets $B_k\subseteq \IR^q$
(here we associate probabilities $\nu$
on the set $\{1, \ldots, q\}$ with vectors in $\IR^q$ and define
$H(\nu|\rho)= H(\nu)=\infty$, if $\nu \in \IR^q$ does not have non-negative components summing to
1).
Together with the above mentioned LDP for the components $M_N(k)$ and the assumption that $|S_k|/N$
 converges to $\gamma_k$ as $N \to \infty$, this observation implies that the matrix $M_N$ under
$\overline \rho_N $ obeys an LDP with speed $N$ and rate function
\begin{equation}\label{rate_ind}
I(\nu) := \sum_{k=1}^s \gamma_k H(\nu_k|\rho)=\log q+\sum_{k=1}^s\gamma_kH(\nu_k)
\end{equation}
Here $\nu:=(\nu_k)_{1\le k\le s}\in M(s\times q)$ and the $\nu_k$ are probabilities on $\{1,
\ldots, q\}$, otherwise $I(\nu)$ is defined to be $\infty$. Thus we have seen
\begin{proposition}\label{prop ldp ind}
Under the measure $\overline \rho_N $ the matrix valued random variable $M_N$ obeys an LDP with
speed $N$ and
rate function $I$ given by \eqref{rate_ind}.
\end{proposition}
Proposition \ref{prop ldp ind} together with the representation of our Hamiltonian in terms of the
matrix $M_N$ \eqref{eqn:HamiltonianAsScalarProduct} immediately yields an LDP for $M_N$ under the
Gibbs measure $\mu_{N}$.
\begin{thm}\label{theo LDP gibbs}
Under the Gibbs measure $\mu_{\alpha,\beta, N}$ the matrix valued random variable $M_N$ obeys an LDP
with speed $N$ and rate function $J = J_{\alpha,\beta}$
\begin{align}\label{rate gibbs}
\begin{split}
J(\nu) \coloneqq{}&-\big[\frac{1}{2}\skal{\Gamma  \nu, \Gamma  \nu}_{A}-I(\nu)\big]
+\sup_\mu\big[\frac{1}{2}\skal{\Gamma  \mu, \Gamma \mu}_{A}-I(\mu)\big]\\
={}&-\Big[ \frac{\beta}{2} \sum_{c = 1}^q \sum_{k=1}^s \gamma_k^2 \nu_{k,c}^2 +
\frac{\alpha}{2} \sum_{c = 1}^q \sum_{k \neq k'} \gamma_k \gamma_{k'} \nu_{k,c}\, \nu_{k',c}
-I(\nu)\Big]\\
&\qquad +
\sup_{\mu} \Big[ \frac{\beta}{2} \sum_{c = 1}^q \sum_{k=1}^s \gamma_k^2 \mu_{k,c}^2 +
\frac{\alpha}{2} \sum_{c = 1}^q \sum_{k \neq k'} \gamma_k \gamma_{k'} \mu_{k,c}\, \mu_{k',c}
-I(\mu)\Big].
\end{split}
\end{align}
Here $\Gamma$ is the $s\times s$ diagonal matrix with $(\Gamma)_{kk}=\gamma_k$  and
$\nu:=(\nu_k)_k$ and $\mu:=(\mu_k)_k$ are $s\times
q$--matrices and the $\nu_k$ and $\mu_k$ are probabilities on $\{1, \ldots, q\}$, otherwise
$J(\nu)$ is defined to be $\infty$.
\end{thm}

\begin{proof}
Starting from \eqref{eqn:HamiltonianAsScalarProduct} we write
\[
H_N(\omega) = -\frac{N}{2} \skal{(N^{-1}\Gamma_N)  M, (N^{-1}\Gamma_N) M}_{A}
\]
and so the assumption that $|S_k|/N\to \gamma_k$ for each $k$, Proposition \ref{prop ldp ind},
together with Varadhan's Lemma (\cite[Theorem III.13]{dH00}) and the tilted LDP (\cite[Theorem
III.17]{dH00}) in the version of \cite[Lemma 2.1]{KLSS19} show the result.
\end{proof}
It may be more convenient to reformulate the LDP above in terms of an LDP for the matrix $M'_N$
with entries
$$
m'_{k,c}:=m'_{k,c}(\omega):=\frac 1 {N} \sum_{i \in S_k} \eins_{\omega_i=c},
$$
so that asymptotically, $M'_N\approx\Gamma M_N$. Then, of course, the relevant matrices are the matrices
$\nu':=\Gamma\nu=(\gamma_k\nu_{kc})_{kc}$ where $\nu$ are the matrices appearing in Theorem \ref{theo LDP gibbs}.
Using \eqref{rate_ind} and calculating
 \[
  \sum_{k = 1}^s \sum_{c = 1}^q \nu_{k,c}' \log \nu_{k,c}'=I(\nu)-\log q+H(\gamma)
 \]
where of course $\gamma=(\gamma_k)$, we can identify the new rate function.
Indeed, as the term $H(\gamma)-\log q$ is independent of $\nu'$, we can reformulate the LDP as follows.

\begin{thm}\label{theo LDP gibbs2}
Under the Gibbs measure $\mu_{\alpha,\beta,N}$ the matrix valued random variable $M'_N$ obeys an LDP
with speed $N$ and rate function
\[
  J'(\nu) = \begin{cases}
  - \big[ \frac{1}{2} \skal{\nu, \nu}_{A} - \sum_{k = 1}^s \sum_{c = 1}^q \nu_{k,c}
\log \nu_{k,c} \big] \\\qquad  +\sup_{\mu \in C(\gamma)} \Big[ \frac{1}{2} \skal{\mu,
\mu}_{A} - \sum_{k = 1}^s \sum_{c = 1}^q \mu_{k,c} \log \mu_{k,c}  \Big] & \nu \in
C(\gamma)\\
  \infty & \nu \notin C(\gamma)
  \end{cases}
\]
where $C(\gamma) = \{ \mu=(\mu_{kc}) \in M(s \times q) : \mu_{kc}\ge0\text{ and } \sum_{c = 1}^q
\mu_{k,c} = \gamma_k \text{ for all }k\}$.
\end{thm}

Note that every matrix $\mu\in C(\gamma)$ can actually be considered as a probability distribution
on $[sq]$ and the term $-\sum_{k \in [s]} \sum_{c \in [q]} \mu_{k,c} \log(\mu_{k,c})$ is its
entropy. However, the set $C(\gamma)$ places restrictions on the mass that can be placed on every
block.

\section{Equilibria for uniform block sizes}

We will now try to find the limit distributions of the matrix valued random variables $M_N$ and $M_N'$ under the sequence of Gibbs measure $\mu_{\alpha,\beta,N}$.
An LDP as in Theorem \ref{theo LDP gibbs} or Theorem \ref{theo LDP gibbs2} is, in principle, of
course able to determine these limit distributions. 
Indeed, they are given by the minima of the
corresponding rate functions.

\begin{cor}\label{cor:equal block sizes limit}
The weak limit points of the sequence $(M_N)$
under the sequence of Gibbs measures $\mu_{N,\alpha,\beta}$ are given by the minima of $J(\cdot)$ and
the weak limit points of the sequence $(M'_N)$
under the sequence of Gibbs measures $\mu_{N,\alpha,\beta}$ are given by the minima of $J'(\cdot)$
\end{cor}
\begin{proof}
This is actually folklore in large deviation theory and not difficult to prove, when realizing that
the upper bound in an LDP implies that any measurable subset of $\R$ whose closure does not contain
a minimum of the rate function has a probability that converges to $0$.
\end{proof}

We will, in the sequel, determine the minima of $J'$. From here, of course, it is also obvious, what the minima of $J$ are.
We start with the observation that
the minimum points of $J'$ are the maximum points of
\begin{align}\label{eq:defG}
G(\mu)\coloneqq{}& \frac{1}{2} \skal{\mu,
\mu}_{A} - \sum_{k = 1}^s \sum_{c = 1}^q \mu_{k,c} \log \mu_{k,c}\nonumber \\
={}& \sum_{c=1}^q \sum_{k=1}^s  \frac{\beta}{2} \mu_{kc}^2 + \sum_{c=1}^q\sum_{
 k' \neq k}^s\frac{\alpha}{2} \mu_{kc}\mu_{k'c}
- \sum_{c=1}^q\sum_{k=1}^s  \mu_{kc} \log \mu_{kc} ,
\end{align}
where $\mu=(\mu_{kc})\in C(\gamma)$ and we have set $0 \log 0:= 0$.

\begin{lemma}\label{strictpos}
$G$ attains its maximum on the set
$$C^+(\gamma):= \{\mu \in C(\gamma): 0 < \mu_{kc}<1 \mbox{ for all } 1\le k \le s, 1 \le c \le
q\}.$$
\end{lemma}
\begin{proof}
Suppose one of $\mu$'s entries equals $0$, without loss of generality $\mu_{11}=0$. Then, there is
$2 \le i \le q$ such that $\mu_{1i} \ge \gamma_1/(q-1)$. Note that $G$ is the sum of a polynomial
of degree two in the $\mu_{kc}$'s and $-\sum_{k=1}^s \sum_{c=1}^q \mu_{kc} \log \mu_{kc}$. Now $-t
\log t$ has derivative infinity at $0$. Hence, for $\vep>0$ small enough, we have $G(\mu)< G(\mu')$
where $\mu'$ is the matrix that we obtain from $\mu$, if we replace $\mu_{11}$ by $\mu_{11}'=\vep$,
$\mu_{1i}$ by $\mu_{1i}'=\mu_{1i}-\vep$ and leave the other entries unaltered. This shows the claim.
\end{proof}

Let us now apply the method of Lagrange multipliers to find the maximum points of $G$. Let
$\lambda=(\lambda_1, \lambda_2, \ldots, \lambda_s)$. We then need to find the critical points of
$$
L(\mu,\lambda)= G(\mu)-\sum_{k=1}^s \lambda_k \left(\sum_{c=1}^q \mu_{kc}-\gamma_k\right).
$$
Differentiating with respect to the $\mu_{kc}$, $1 \le k \le s$, $1 \le c \le q$ gives the
following set of equations
\begin{equation}\label{eq:partial}
0=\partial_{\mu_{kc}} G(\mu)-\lambda_k= \beta \mu_{kc}+\alpha\sum_{
\substack{k'=1\\ k' \neq k}}^s \mu_{k'c}-
\log \mu_{kc} -1 -\lambda_k.
\end{equation}
Summing these equations over all $c$ yields
$$
q(1+\lambda_k)=\beta\gamma_k+\alpha\sum_{\substack{k'=1\\k'\neq k}}^s
\gamma_{k'}-\sum_c \log\mu_{kc},
$$
and plugging this into \eqref{eq:partial} we finally arrive at our system of critical equations
\begin{equation}\label{gls}
 \beta \Big(\mu_{kc}-\frac{\gamma_k}{q}\Big)+\alpha \sum_{\substack{k'=1\\k'\ne k}}^s
\Big(\mu_{k'c}-\frac{\gamma_{k'}}{q}\Big)
=\log \frac{\mu_{kc}}{\sqrt[q]{\prod_d \mu_{kd}}}.
\end{equation}
that any maximum point has to solve. Let us rephrase the value of $G$ in critical points
by multiplying \eqref{gls} by $\mu_{kc}$ and summing over $c$:
$$
\sum_c (\beta \mu_{kc}^2+\alpha
\sum_{\substack{k'=1\\k'\ne k}}^s \mu_{kc}\mu_{k'c}-2\mu_{kc}\log\mu_{kc})=
\frac{\gamma_k}{q}(\beta\gamma_k+\alpha(1-\gamma_k))-\sum_c\Big(\frac{\gamma_k}{q}+
\mu_{kc}\Big)\log\mu_{kc}.
$$
Hence, if $\mu^{\mathrm{crit}}$ is a critical point of $G$, we can write $G(\mu^{\mathrm{crit}})$
as
\begin{equation}\label{eq:Gcrit}
G(\mu^{\mathrm{crit}})=
\frac{1}{2q}((\beta-\alpha)\|\gamma\|^2+\alpha)-\frac{1}{2q}\sum_{k=1}^s\sum_{c=1}^q
 (\gamma_k+q\mu^{\mathrm{crit}}_{kc})\log \mu^{\mathrm{crit}}_{kc}.
\end{equation}

Next we will see that of all critical points only those where all rows of $\mu$ have the same,
e.\,g.\
an increasing order, are relevant.

\begin{lemma}\label{genumord}
For $\mu\in C^+(\gamma)$ with increasing rows $\mu_{k1}\le\ldots\le \mu_{kq}$,
$1\le k\le s$, we have
\begin{equation}\label{umordnung}
 \sum_{k'\ne k}\sum_{c=1}^q \mu_{kc}\mu_{k'c}\ge \sum_{k'\ne k}\sum_{c=1}^q
 \mu_{k \sigma_k(c)}\mu_{ k'\sigma_{k'}(c)}
\end{equation}
for all $s$--tuples $(\sigma_k)_k$ of permutations $\sigma_k\in S_q$.

In particular, for $\alpha>0$, the function $G$ can only be maximal in a point
$\mu'$, if the rows of $\mu'$ are ordered in the same way, i.\,e.\ if there is a $\sigma\in S_q$
such
that
$$\mu'_{k\sigma(1)}\le\ldots\le \mu'_{k\sigma(q)}$$ for all
$1\le k\le s$.
\end{lemma}

\begin{proof}
Recall the rearrangement inequality \cite{inequalities}, Theorem 368: When $x_1 \le \ldots \le x_n$
and  $y_1 \le \ldots \le y_n$ are sequences of real numbers, then for every permutation $\pi \in
S_n$ one has
$$
\sum_{i=1}^nx_i y_{\pi(i)}\le x_1 y_1+ x_2 y_2+\ldots x_n y_n
$$
and the inequality is strict if there are indices $j<j'$ with $x_j<x_{j'}$ and
$y_{\pi(j)}>y_{\pi(j')}$.

Applying this to row $k$ and $k'$ of $\mu$ gives
\[
 \sum_{c=1}^q\mu_{kc}\mu_{k'c}\ge\sum_{c=1}^q\mu_{k\sigma_k(c)}\mu_{k'
 \sigma_{k'}(c)}
\]
for every two permutations $\sigma_k,\sigma_{k'} \in S_q$.
Summing over all $k'\neq k$ yields \eqref{umordnung}.

In particular, $\mu'$ is not a maximum point, if we have two rows $k\ne k'$ and
indices $j<j'$ with $\mu'_{kj}<\mu'_{kj'}$ and $\mu'_{k'j}>\mu'_{k'j'}$.
\end{proof}

So we can and will assume in the following, that all rows of a critical
point $\mu$ of $G$ are increasing. The next lemma determines the structure of a critical $\mu$.

\begin{lemma}\label{lem:same columns}
Let $\mu$ be a critical point of $G$.
\begin{enumerate}
 \item If $\mu_{kc}=\mu_{kc'}$ for a $k$ and $c\ne c'$ then
$\mu_{k'c}=\mu_{k'c'} $ for all $1\le k'\le s$.
\item Each row of $\mu$ has at most two different entries.
\end{enumerate}
\end{lemma}

\begin{proof}
Substracting \eqref{gls} for $c$ from the equation for $c'$ yields
$$
\alpha\sum_{k'\ne k}\mu_{k'c}=\alpha\sum_{k'\ne k}\mu_{k'c'}
$$
and thus, by increasing order, $\mu_{k'c}=\mu_{k'c'}$ for all $1\le k'\le s$.
This is the first claim.

For two different columns $c\ne c'$ we obtain from \eqref{gls}
\[
(\beta-\alpha)(\mu_{kc'}-\mu_{kc})+\alpha\sum^s_{k'=1}(\mu_{k'c'}-\mu_{k'c})=
\log \mu_{kc'}-\log \mu_{kc}.
\]
Now if we had three columns $c<c'<c''$ with $\mu_{kc}<\mu_{kc'}<\mu_{kc''}$ for one row $k$
(and hence for all due to the first part of this lemma) we would have
\begin{equation}\label{eq:after_sum}
\frac{1}{\alpha}=\sum_{k=1}^s\frac{\mu_{kc'}-\mu_{kc}}{\alpha\sum^s_{k'=1}(\mu_{k'c'}-\mu_{k'c})}=
\sum_{k=1}^s\Big(  \frac{\log  \mu_{kc'}-\log \mu_{kc}}{\mu_{kc'}-\mu_{kc}}-(\beta-\alpha)
\Big)^{-1}
\end{equation}
and the same equation for the pair $c, c''$.
However, every summand on the right of that latter equation would be larger (and positive)
than the corresponding summand in \eqref{eq:after_sum} by concavity of the logarithm.
Hence we have a contradiction.
\end{proof}

So far we have proved, that according to Lemma \ref{strictpos}, Lemma \ref{genumord} and Lemma
\ref{lem:same columns} we can constrain our search for maximum
points of $G$ to matrices $\mu$ with positive
entries, increasingly ordered rows and having at most two different columns. Taking into account
that the entries in row $k$ sum up to $\gamma_k$ we see that the largest column $\mu^+$ (component
by component) of $\mu$
together
with the number $1\le r\le q$ of columns equal to $\mu^+$ is all the information we need to build
up $\mu$.
So either $\mu$ has $q$ identical columns $\gamma/q$ or (the increasingly ordered) $\mu$ reads
\[
 \mu=(\underbrace{\mu^-\ldots\mu^-}_{q-r
 }\underbrace{\mu^+\ldots \mu^+}_{r})
\]
with $\mu^-=(\gamma-r\mu^+)/(q-r)$ for some $1\le r\le q-1$ and $\mu^+_k>\gamma_k/q>\mu^-_k$ for
all $k$.

\medskip

At this stage the authors do not know a way to proceed with the case of an arbitrary vector
$\gamma$. However, a proof for the case of asymptotically equal block sizes $\gamma_k=1/s$ for all
$1\le k\le s$ is readily accomplished.

\begin{proposition}
 Let $\gamma_k=1/s$ for all $1\le k\le s$ and $Q\in M(s\times q)$ with identical entries $1/sq$.
If a critical $\mu$ with $G(Q)\le G(\mu)$ does not have identical rows then $\mu$ is not a maximum
point.
\end{proposition}

\begin{proof}
 Recall that the system of critical equations \eqref{gls} reads
 \[
  \alpha\sum_{k'=1}^s\Big(\mu^+_{k'}-\frac{1}{sq}\Big)=\frac{q-r}{q}\log \frac{s(q-r)\mu_k^+}{1-sr\mu^+_k}
  -(\beta-\alpha)\Big(\mu^+_k-\frac{1}{sq}\Big)
 \]
and regard the right hand side of this equation as a function of $\mu_k^+$, say $\psi(\mu_k^+)$.
 Now if the largest entry of the vector $\mu^+$ occurs in line $K$ (and perhaps somewhere else) but
not in every line, then
\[
 \psi(\mu^+_K)-\alpha\Big(s\mu^+_K-\frac{1}{q}\Big)<0.
\]
Since $\psi(t)$ diverges to $+\infty$ when $t$ approaches $1/sr$ from below we can find a
$t_0>\mu_K^+$ with $\psi(t_0)=\alpha(st_0-1/q)$. Now building a matrix $\nu$ by taking instead
of $r$ columns equal to $\mu^+$ just $r$ columns with identical entries $t_0$ and completing the
matrix with $q-r$ columns with identical entries $(1-srt_0)/(s(q-r))$ then clearly $\nu$ is (well
defined and) a critical point. So we just have to prove that $G(\mu)<G(\nu)$.

To that end observe that according to \eqref{eq:Gcrit} the value of $G$ in
critical points $p$ is given up to constants as $\sum_kw(p^+_k)$
 with $w:(0,1/(sr))\to\R$
 \[
  w(x)=-((q-r)+q(1-srx))\log\frac{1-srx}{s(q-r)}-r(1+sqx)\log x.
 \]
Calculating
 \[
  w'(x)=srq\log\frac{1-srx}{s(q-r)x}+r\frac{sqx-1}{x(1-srx)}
 \]
and
\[
 w''(x)=\frac{r(sqx-1)(2srx-1)}{x^2(srx-1)^2}
\]
we see that if $q>2r$ then the graph of $w$, coming from $+\infty$ at the vertical asymptotic line
$x=0$, has a saddle point at $x=1/sq$
changing from bending to the left to bending to the right. It decreases to the second inflection
point at $x=1/(2sr)$, passes, now bending to the left again, the unique minimum point at say
$\xi$ and disappears to $+\infty$ approaching the vertical line at $x=1/(sr)$.
In this case clearly $w(\mu^+_k)\le\max(w(\mu^+_K),w(1/sq))$ for all $k$ since
$1/sq<\mu^+_k\le\mu^+_K$. Now $w(\mu^+_K)<w(1/sq)$ would imply the contradiction
\[
G(\mu)=\sum_kw(\mu^+_k)<sw(1/sq)=G(Q)
\]
so we have $w(1/sq)\le w(\mu^+_K)$ and therefore $\mu^+_K>\xi$ which means
$w(t_0)>w(\mu^+_K)$ and so
\[
 G(\nu)=sw(t_0)>sw(\mu^+_K)\ge G(\mu).
\]
If $q\le 2r$ then $w$ is strictly increasing on $[1/sq,1/sr)$ so that
\[
 G(\nu)=sw(t_0)>\sum_kw(\mu^+_k)=G(\mu).
\]
\end{proof}

Wrapping up what we have seen, we state

\begin{proposition}\label{prop:wrappingup}
Let $\gamma\in \R^s$ have identical entries $1/s$.
The function $G=:G^{\mbox{bP}}$ in the block spin Potts model is maximal on
$C(\gamma)$ if its rows are identical and equal to a maximizer of the corresponding target function
\begin{equation}\label{eq:Potts}
 G^{\mbox{P}}(v)=\frac{1}{2s}(\beta+(s-1)\alpha)\sum_{c=1}^q v_c^2-\sum_{c=1}^qv_c\log v_c
\end{equation}
on the set $V:=\{v=(v_1,\ldots,v_q)\mid \sum_cv_c=1,\,v_c>0\}$ in the Potts model.
\end{proposition}

\begin{proof}
The maximum of $G^{\mbox{bP}}$ on $C(\gamma)$ is attained on the subset of matrices with identical rows
taken from $\{v/s\mid v\in V\}$ and the value is equal to
\[
 \frac{\beta+(s-1)\alpha}{2s}\sum_{c=1}^qv_c^2-\sum_{c=1}^qv_c\log \frac{v_c}{s}.
\]
However, up to a minus sign and ignoring the summand $\log s$ this is the free energy
functional in a Potts model at inverse temperature $(\beta+(s-1)\alpha)/s$.
\end{proof}

The following theorem hence follows from the results in
\cite{Kesten_Potts,EllisWang-LimitTheoremsForTheEmpiricalVectorOfTheCWPottsModel}, where the
critical temperature and the behaviour of the Potts model is computed.

\begin{thm}\label{theo:phasetransition}
Consider the block spin Potts model in the asymptotically uniform case $\gamma_k = s^{-1}$. Denote
by $\zeta_q:=[2\frac{q-1}{q-2}\log(q-1)]$ the critical inverse temperature in the $q$-color Potts
model, and let $g:= \frac{\beta+(s-1)\alpha}s$.
Then the q-color block spin Potts model has a phase transition. More precisely, if
$g < \zeta_q$, then the distribution of $M'_N$ under the Gibbs measure
$\mu_{N,\alpha,\beta}$ concentrates in a unique point, the matrix with all entries identical to
$1/(sq)$.

To describe the  ``low temperature'' behavior define the function $\varphi: [0,1]\to \IR^q$:
$$
\varphi(t):= \Big(\frac{1+(q-1)t}{sq}, \frac{1-t}{sq}, \ldots, \frac{1-t}{sq}\Big)
$$
and let $u(g)$ be the largest solution of the equation
$$
u=\frac{1-e^{-g u}}{1+(q-1)e^{-g u}}.
$$
Finally let $n^1(g):= \varphi(u(g))$ and $n^i(g)$ be $n^1(g)$ with the $i'th$
and the first coordinate interchanged, $i=2, \ldots, q$. Let $\nu^i(g)$ be the matrix with all
rows identical to $n^i(g)$ and $Q$ be the matrix with all entries identical to $1/(qs)$.

Then, if $g > \zeta_q$ the distribution of $M_N'$ under the Gibbs measure
$\mu_{N,\alpha,\beta}$ concentrates in a (uniform) mixture of the Dirac measures in $\nu^1(g),
\ldots
\nu^q(g)$.

At $g =\zeta_q$ the limit points of  $M_N'$ under the Gibbs measure
$\mu_{N,\alpha,\beta}$ are
$Q$ and $\nu^1(g), \ldots \nu^q(g)$.
\end{thm}

\section{Logarithmic Sobolev inequalities}

In this section, we present logarithmic Sobolev inequalities for block spin Potts models. Logarithmic Sobolev inequalities are frequently used in concentration of measure theory (yielding non-asymptotic fluctuation and deviation results), for example, where they form the core of the well-known entropy method. See the monographs \cite{Led01} or \cite{BLM13} for an overview on these topics. Recently, logarithmic Sobolev inequalities for various type of finite spin systems have been established, starting with the Ising model in \cite{GSS18}, to be followed by exponential random graph, random coloring and hard-core models in \cite{SS18}. Here we continue this line of research by considering block spin Potts models.

Let us first recall some basic notions. For $\omega = (\omega_i)_{i \in S} \in \{1, \ldots, q\}^S$ and $i \in S$, we write $\omega_{i^c} := (\omega_j)_{j \ne i}$. Moreover, for any function $f \colon \{1, \ldots, q\}^S \to \mathbb{R}$, we define a certain ``difference operator'' by
\[
\lvert \mathfrak{d} f\rvert(\omega) = \Big(\sum_{i \in S} \int (f(\omega) - f(\omega_{i^c}, \omega_i'))^2 d\mu_N(\omega_i' \mid \omega_{i^c})\Big)^{1/2},
\]
where $\mu_N(\cdot \mid \omega_{i^c})$ denotes the regular conditional probability. The integrals may be regarded as a kind of ``local variance'' in the respective coordinate. The difference operator $\mathfrak{d}$ is a well-known object, and the integral of $|\mathfrak{d}f|^2$ with respect to $\mu_N$ can be regarded as a Dirichlet form (cf. \cite[Remark 2.2]{GSS18}). Finally, for any non-negative function $f$, $\mathrm{Ent}_{\mu_N} (f) := \int f \log(f) d\mu_N - \int f d\mu_N \log(\int f d\mu_N)$ denotes the entropy.

We now have the following log-Sobolev-type inequalities.

\begin{thm}\label{PottsLSIs}
	Assume that $2q\beta e^\beta < 1$.
	\begin{enumerate}
		\item For $N$ large enough, $\mu_N$ satisfies a logarithmic Sobolev inequality with respect to $\mathfrak{d}$. That is, for any function $f \colon \{1, \ldots, q\}^S \to \mathbb{R}$ we have
		\begin{equation}\label{PottsLSI}
			\mathrm{Ent}_{\mu_N}(f^2) \le 2\sigma_1^2 \int \lvert \mathfrak{d}f \rvert^2 d\mu_N.
		\end{equation}
		\item For $N$ large enough and any function $f \colon \{1, \ldots, q\}^S \to \mathbb{R}$ we have
		\begin{equation}\label{PottsmLSI1}
			\mathrm{Ent}_{\mu_N}(e^f) \le \sigma_2^2 \sum_{i \in S} \int \mathrm{Cov}_{\mu(\cdot \mid \omega_{i^c})}(f(\omega_{i^c}, \cdot), e^{f(\omega_{i^c},\cdot)}) d\mu(\omega).
		\end{equation}
		\item For $N$ large enough and any function $f \colon \{1, \ldots, q\}^S \to \mathbb{R}$ we have
		\begin{equation}\label{PottsmLSI2}
			\mathrm{Ent}_{\mu_N}(e^f) \le \frac{\sigma_3^2}{2} \int \lvert\mathfrak{d} f\rvert^2 e^f d\mu_N.
		\end{equation}
	\end{enumerate}
	Here, $\sigma_1, \sigma_2, \sigma_3 > 0$ are constants which depend on $\beta$ and $q$ only.
\end{thm}

It is possible to give explicit values for $\sigma_1, \sigma_2, \sigma_3$ which depend on quantities which stem from the conditional probabilities $\mu_N(\cdot |\omega_{i^c})$. For details, see Remark \ref{constants}, where we also comment on the condition on $N$.

Both \eqref{PottsmLSI1} and \eqref{PottsmLSI2} are also known as modified logarithmic Sobolev inequalities. In fact, if $\mathfrak{d}f$ is replaced by another difference operator which satisfies the chain rule (for instance, the ordinary Euclidean gradient), then \eqref{PottsLSI} and \eqref{PottsmLSI2} are equivalent, but obviously, this is not true for $\mathfrak{d}f$.

Let us briefly discuss the results from Theorem \ref{PottsLSIs}, including possible applications.  From \eqref{PottsLSI} and \eqref{PottsmLSI2}, we may derive concentration of measure bounds by various techniques. For instance, based on \eqref{PottsLSI} we obtain $L^p$ norm inequalities for any function $f \in L^\infty(\mu_N)$ and subsequently concentration results, cf.\ \cite[Proposition 2.4, Theorem 1.5]{GSS18}. Moreover, \eqref{PottsmLSI2} gives rise to subgaussian tail bounds for Lipschitz-type functions $f$ (in the sense of $|\mathfrak{d}f| \le 1$) by applying the Herbst argument, see e.\,g.\ \cite{SS19} (where also slightly more advanced situations are discussed, cf.\ Section 2.4). Recently, in \cite{APS20} it has been established that \eqref{PottsmLSI2} moreover gives rise to $L^p$ bounds (via Beckner inequalities) as well. Especially for the case of spin systems and Glauber dynamics, cf.\ the discussion in Section 4.3 therein.

To consider a simple example, let
\[
T_{k,c}(\omega) := \sum_{i \in S_k} \eins_{\omega_i = c}
\]
denote the number of vertices in the block $S_k$ which have colour $c$. It is easy to check that $|\mathfrak{d}T_{k,c}|^2 \le |S_k|$, and therefore, using \cite[Equation (1.2)]{SS19}, we immediately obtain that
\[
\mu_N(|T_{k,c} - \mu_N(T_{k,c})| \ge t) \le 2 \exp \Big(-\frac{t^2}{2|S_k|\sigma_3^2}\Big),
\]
where $\mu_N(T_{k,c}) := \int T_{k,c} d\mu_N$. Note that for $N$ large, this probability approaches $2 \exp(-t^2/(2N\gamma_k\sigma_3^2))$.

Furthermore, a special class of non-Lipschitz functions for which concentration bounds based on inequalities of type \eqref{PottsLSI} and \eqref{PottsmLSI2} have been established in recent years are so-called multilinear polynomials, i.\,e.\ polynomials which are affine with respect to each variable. See for instance \cite[Theorem 5]{GSS18b} and \cite[Corollary 5.4]{APS20}.

Finally, note that \eqref{PottsmLSI1} is frequently used in the context of Markov processes, and it can be shown to be equivalent to exponential decay of the relative entropy along the Glauber semigroup (cf.\ e.\,g.\ \cite{BT06} or \cite{CMT15}). As shown in \cite[Theorem 2.2]{SS18}, it moreover implies that the associated Glauber dynamics is rapidly mixing, i.\,e.\ its mixing time is of order $O(N\log N)$. This complements the results from \cite{BGP16}, where a different situation was considered (the usual Potts model without blocks but on graphs with fixed maximal degree).

The remaining part of this section is devoted to the proof of Theorem \ref{PottsLSIs}. To convey the basic idea, recall that for product measures $\otimes_{i=1}^n \mu_i$, the entropy functional tensorizes in the sense that
\[
\mathrm{Ent}_{\otimes_{i=1}^n \mu_i} (f) \le \sum_{i=1}^{N} \int \mathrm{Ent}_{\mu_i}(f) d\mu,
\]
and therefore, proving logarithmic Sobolev inequalities reduces to controlling each coordinate separately, i.\,e.\ a ``one-dimensional'' case. For non-product measures, an inequality of this type no longer holds true, but if the dependencies are sufficiently weak, an analogue can be shown which is known as an \emph{approximate tensorization property}. For probability spaces with finitely many atoms, a suitable criterion to establish approximate tensorization was introduced in \cite{Ma15}, which we will exploit in the sequel.

\begin{proposition}\label{PottsApproxTen}
	Assume that $2q\beta e^\beta < 1$. For $N$ large enough, the approximate tensorization property of entropy holds, i.\,e.
	\[
	\mathrm{Ent}_{\mu_N}(f) \le C \sum_{i \in S} \int \mathrm{Ent}_{\mu_N(\cdot | \omega_{i^c})}(f(\omega_{i^c}, \cdot)) d\mu_N(\omega)
	\]
	with $C$ depending on $\beta$ and $q$ only.
\end{proposition}

\begin{proof}
	As pointed out before, the proof of Proposition \ref{PottsApproxTen} works by applying Marton's approximate tensorization result \cite{Ma15} in the slightly rewritten and corrected form stated in \cite[Theorem 4.1]{SS18}. Essentially, we need to control the conditional probabilities $\mu_N(\cdot | \omega_{i^c})$, which we rewrite in the sequel, generalizing the case of the usual Potts model without blocks as in \cite[Proposition 2.16]{Sin18}. Recalling the matrix $B=(b_{k,c})$ from \eqref{matrixB}, we fix two sites $i,j \in S$ and define
	\[
	b_{k,ij,c} \coloneqq \sum_{\substack{\nu \in S_k \\ \nu \notin \{i,j\}}} \eins_{\omega_\nu = c}.
	\]
	Then, the Hamiltonian may be decomposed as follows:
	\begin{align*}
		-H_N(\omega) &= \frac{\beta}{2N} \sum_{c=1}^{q} \sum_{k=1}^{s} (b_{k,ij,c} + \eins_{i \in S_k, \omega_i = c} + \eins_{j \in S_k, \omega_j = c})^2\\
		&\quad + \frac{\alpha}{2N} \sum_{c=1}^{q} \sum_{k=1}^{s} \sum_{k' : k' \ne k} (b_{k,ij,c} + \eins_{i \in S_k, \omega_i = c} + \eins_{j \in S_k, \omega_j = c})\\
		&\hspace{4cm}(b_{k',ij,c} + \eins_{i \in S_{k'}, \omega_i = c} + \eins_{j \in S_{k'}, \omega_j = c}).
	\end{align*}
	Now, writing
	\begin{align*}
		&(b_{k,ij,c} + \eins_{i \in S_k, \omega_i = c} + \eins_{j \in S_k, \omega_j = c})^2
		= b_{k,ij,c}^2 + 2 b_{k,ij,c}(\eins_{i \in S_k, \omega_i = c} + \eins_{j \in S_k, \omega_j = c})\\
		&\hspace{4cm}+ \eins_{i \in S_k, \omega_i = c} + \eins_{j \in S_k, \omega_j = c} + 2 \cdot \eins_{i \in S_k, \omega_i = c}\eins_{j \in S_k, \omega_j = c}
	\end{align*}
	and summing over $c$ and $k$, we obtain
	\begin{align*}
		&\sum_{c=1}^{q} \sum_{k=1}^{s} (b_{k,ij,c} + \eins_{i \in S_k, \omega_i = c} + \eins_{j \in S_k, \omega_j = c})^2\\
		= \ &\sum_{c=1}^{q} \sum_{k=1}^{s} b_{k,ij,c}^2 + 2 \sum_{k=1}^s b_{k,ij,\omega_i} \eins_{i \in S_k} + 2 \sum_{k=1}^s b_{k,ij,\omega_j} \eins_{j \in S_k} + 2 + 2 \cdot \eins_{i \sim j, \omega_i = \omega_j}.
	\end{align*}
	Similarly,
	\begin{align*}
		&(b_{k,ij,c} + \eins_{i \in S_k, \omega_i = c} + \eins_{j \in S_k, \omega_j = c})(b_{k',ij,c} + \eins_{i \in S_{k'}, \omega_i = c} + \eins_{j \in S_{k'}, \omega_j = c\}})\\
		= \ &b_{k,ij,c}b_{k',ij,c} + b_{k,ij,c}(\eins_{i \in S_{k'}, \omega_i = c} + \eins_{j \in S_{k'}, \omega_j = c}) + b_{k',ij,c} (\eins_{i \in S_k, \omega_i = c} + \eins_{j \in S_k, \omega_j = c})\\
		&+ (\eins_{i \in S_k, \omega_i = c} + \eins_{j \in S_k, \omega_j = c})(\eins_{i \in S_{k'}, \omega_i = c} + \eins_{j \in S_{k'}, \omega_j = c}),
	\end{align*}
	leading to
	\begin{align*}
		&\sum_{c=1}^{q} \sum_{k=1}^{s} \sum_{k' : k' \ne k} (b_{k,ij,c} + \eins_{i \in S_k, \omega_i = c} + \eins_{j \in S_k, \omega_j = c})(b_{k',ij,c} + \eins_{i \in S_{k'}, \omega_i = c} + \eins_{j \in S_{k'}, \omega_j = c})\\
		= \ &\sum_{c=1}^{q} \sum_{k=1}^{s} \sum_{k' : k' \ne k} b_{k,ij,c}b_{k',ij,c} + 2 \sum_{k=1}^s b_{k,ij,\omega_i} \eins_{i \notin S_k} + 2 \sum_{k=1}^s b_{k,ij,\omega_j} \eins_{j \notin S_k} + 2 \cdot \eins_{i \nsim j, \omega_i = \omega_j}.
	\end{align*}
	
	Altogether, we arrive at the representation
	\begin{align*}
		-H_N(\omega)
		= \ &\frac{\beta}{2N} \sum_{c=1}^{q} \sum_{k=1}^{s} b_{k,ij,c}^2 + \frac{\beta}{N} \sum_{k=1}^s b_{k,ij,\omega_i} \eins_{i \in S_k} + \frac{\beta}{N} \sum_{k=1}^s b_{k,ij,\omega_j} \eins_{j \in S_k} + \frac{\beta}{N}\\
		&+ \frac{\beta}{N} \eins_{i \sim j, \omega_i = \omega_j} + \frac{\alpha}{2N} \sum_{c=1}^{q} \sum_{k=1}^{s} \sum_{k' : k' \ne k} b_{k,ij,c}b_{k',ij,c}\\
		&+ \frac{\alpha}{N} \sum_{k=1}^s b_{k,ij,\omega_i} \eins_{i \notin S_k} + \frac{\alpha}{N} \sum_{k=1}^s b_{k,ij,\omega_j} \eins_{j \notin S_k} + \frac{\alpha}{N} \eins_{i \nsim j, \omega_i = \omega_j}.
	\end{align*}
	In particular, the conditional probabilities given $\omega_{i^c}$ can be written as
	\begin{align*}
		\mu_N(c | \omega_{i^c}) &= \frac{\exp(-H_N(\omega_{i^c}, c))}{\sum_{c'=1}^{q}\exp(-H_N(\omega_{i^c}, c'))}\\
		&= \frac{1}{1 + \sum_{c' : c' \ne c}\exp(H_N(\omega_{i^c}, c)- H_N(\omega_{i^c}, c'))}\\
		&= h\big(\sum_{c' : c' \ne c}\exp(H_N(\omega_{i^c}, c)- H_N(\omega_{i^c}, c'))\big),
	\end{align*}
	where $h(x) = 1/(1+x)$, which is $1$-Lipschitz. Here, using the representation derived above, we have
	\begin{align*}
		&H_N(\omega_{i^c}, c)- H_N(\omega_{i^c}, c')\\
		= \ &\frac{\beta}{N} \sum_{k=1}^s (b_{k,ij,c'} - b_{k,ij,c}) \eins_{i \in S_k} + \frac{\beta}{N} (\eins_{\omega_j = c'} - \eins_{\omega_j = c}) \eins_{i \sim j}\\
		&\hspace{2cm}+ \frac{\alpha}{N} \sum_{k=1}^s (b_{k,ij,c'} - b_{k,ij,c}) \eins_{i \notin S_k} + \frac{\alpha}{N} (\eins_{\omega_j = c'} - \eins_{\omega_j = c}) \eins_{i \nsim j}\\
		= \ &\frac{1}{N}\sum_{k=1}^s (b_{k,ij,c'} - b_{k,ij,c}) (\beta \eins_{i \in S_k}+ \alpha \eins_{i \notin S_k}) + \frac{1}{N}(\eins_{\omega_j = c'} - \eins_{\omega_j = c})(\beta \eins_{i \sim j} + \alpha \eins_{i \nsim j}).
	\end{align*}
	
	From this representation, we now derive two facts: first,
	\begin{equation}
		\label{Bed1}
		\min_{i\in S} \min_{\omega \in \{1, \ldots, q\}^S} \mu_N(\omega_i | \omega_{i^c}) \ge \gamma_1
	\end{equation}
	for some $\gamma_1 > 0$ which only depends on $\beta$ but not on $N$ (here we have used that $\alpha < \beta$).
	
	Moreover, we need to control the operator norm $\lVert J \rVert_{2 \to 2}$ of the $N \times N$ matrix $J$ whose entries $J_{ij}$ are given by
	\[
	J_{ij} = \sup d_\mathrm{TV} (\mu_N(\cdot |\omega_{i^c}), \mu_N(\cdot |\sigma_{i^c})),
	\]
	where the sup is taken over all configurations $\omega, \sigma$ which differ at site $j$ only. Obviously, this can be upper bounded by the $\ell^\infty \to \ell^\infty$ norm $\lVert J \rVert_{\infty \to \infty}$. To bound the latter, we fix two such configurations, i.\,e. $\omega_j \ne \sigma_j$ and $\omega_{j^c} = \sigma_{j^c}$. It follows that for any $i \ne j$,
	\begin{align*}
		&d_\mathrm{TV} (\mu_N(\cdot |\omega_{i^c}), \mu_N(\cdot |\sigma_{i^c})) = \frac{1}{2} \sum_{c=1}^{q} \abs{\mu_N(c |\omega_{i^c})- \mu_N(c |\sigma_{i^c})}\\
		= \ &\frac{1}{2} \sum_{c=1}^{q} \Big|h\big(\sum_{c' : c' \ne c}\exp(H_N(\omega_{i^c}, c)- H_N(\omega_{i^c}, c'))\big)\\
		&\hspace{2cm}- h\big(\sum_{c' : c' \ne c}\exp(H_N(\sigma_{i^c}, c)- H_N(\sigma_{i^c}, c'))\big)\Big|\\
		\le \ &\frac{1}{2} \sum_{c=1}^{q} \Big| \sum_{c' : c' \ne c} \Big(\exp(H_N(\omega_{i^c}, c)- H_N(\omega_{i^c}, c')) - \exp(H_N(\sigma_{i^c}, c)- H_N(\sigma_{i^c}, c'))\Big) \Big|\\
		\le \ &\frac{1}{2} \sum_{c=1}^{q} \sum_{c' : c' \ne c} \exp \Big(\frac{1}{N}\sum_{k=1}^s (b_{k,ij,c'} - b_{k,ij,c}) (\beta \eins_{\{i \in S_k\}}+ \alpha \eins_{\{i \notin S_k\}})\Big)\\
		&\hspace{2cm}\Big|\exp\Big(\frac{1}{N}(\eins_{\omega_j = c'} - \eins_{\omega_j = c})(\beta \eins_{i \sim j} + \alpha \eins_{i \nsim j})\Big)\\
		&\hspace{3cm}- \exp\Big(\frac{1}{N}(\eins_{\sigma_j = c'} - \eins_{\sigma_j = c})(\beta \eins_{i \sim j} + \alpha \eins_{i \nsim j})\Big)\Big|\\
		\le \ &\frac{1}{2} e^\beta \sum_{c=1}^{q} \sum_{c' : c' \ne c} \Big|\exp\Big(\frac{1}{N}(\eins_{\omega_j = c'} - \eins_{\omega_j = c})(\beta \eins_{i \sim j} + \alpha \eins_{i \nsim j})\Big)\\
		&\hspace{2cm}- \exp\Big(\frac{1}{N}(\eins_{\sigma_j = c'} - \eins_{\sigma_j = c})(\beta \eins_{i \sim j} + \alpha \eins_{i \nsim j})\Big)\Big|\\
		\eqqcolon \ &e^\beta \sum_{c=1}^{q} \sum_{c' : c' \ne c} I(c, c').
	\end{align*}
	Here, the first inequality follows from the Lipschitz property, in the second inequality we have used that $b_{k,ij,c}(\omega) = b_{k,ij,c}(\sigma)$ for all $c$ since $\omega_{j^c} = \sigma_{j^c}$, and in the last inequality we have used $\alpha < \beta$ once again. Note that whenever $\{c, c' \} \cap \{\omega_j, \sigma_j \} = \emptyset$, we clearly have $I(c, c') = 0$. It follows that the right-hand side can be written as
	\[
	e^\beta \Big(\sum_{c' \ne \omega_j} I(\omega_j, c') + \sum_{c' \ne \sigma_j} I(\sigma_j, c') + \sum_{c \ne \omega_j} I(c, \omega_j) + \sum_{c \ne \sigma_j} I(c, \sigma_j) \Big).
	\]
	By Taylor expansion, assuming $i \sim j$ it is not hard to see that
	\begin{align*}
		I(\omega_j, \sigma_j) &= I(\sigma_j, \omega_j) \le \frac{2\beta}{N} + o((\beta/N)^2),\\
		I(\omega_j, c') &= I(\sigma_j, c') = I(c, \omega_j) = I(c, \sigma_j) \le \frac{\beta}{N} + O((\beta/N)^2),
	\end{align*}
	where in the second line $c, c' \notin \{\omega_j, \sigma_j\}$. The same bounds with $\beta$ replaced by $\alpha$ hold if $i \nsim j$. Altogether, it follows that
	\begin{align*}
		\lVert J \rVert_{\infty \to \infty} = \max_{i\in S} \sum_{j\in S} |J_{ij}| &\le \frac{4(N-1)}{2} e^\beta\Big((q-2)\frac{\beta}{N} + 2\frac{\beta}{N} + O_\beta(N^{-2})\Big)\\
		&\le 2q\beta e^\beta + O_\beta(N^{-1}).
	\end{align*}
	In particular, if $N$ is sufficiently large, then
	\begin{equation}\label{Bed2}
		\lVert J \rVert_{2 \to 2} \le 1 - \gamma_2
	\end{equation}
	for any $\gamma_2 < 1 - 2q\beta e^\beta$.
	
	Combining \eqref{Bed1} and \eqref{Bed2} and applying \cite[Theorem 4.1]{SS18} (which in particular requires $\gamma_2 \in (0,1)$), we arrive at the result.
\end{proof}

\begin{proof}[Proof of Theorem \ref{PottsLSIs}]
	Once we have established Proposition \ref{PottsApproxTen}, this follows immediately from \cite[Proof of Theorem 4.1]{SS18} and \cite[(2.10)--(2.12)]{SS19}.
\end{proof}

\begin{rem}\label{constants}
	By a closer look at the respective proofs, it is possible to give explicit values for the constants appearing in Theorem \ref{PottsLSI} and Proposition \ref{PottsApproxTen} depending on the quantities $\gamma_1$ and $\gamma_2$ from the proof of Proposition \ref{PottsApproxTen}. Indeed, we may set $C = (\gamma_1\gamma_2^2)^{-1}$, $\sigma_2^2 = \sigma_3^2 = C$ and $\sigma_1^2= \log(\gamma_1^{-1})C/\log(4)$.
	
	Moreover, requiring $N$ to be large enough means that $N$ must be so large that $2q\beta e^\beta + O_\beta(N^{-1}) < 1$ in the asymptotics leading to \eqref{Bed2}.
\end{rem}


\end{document}